\documentclass[a4paper,11pt]{article}

\usepackage[utf8]{inputenc} 
\usepackage{enumerate}

\usepackage[T1]{fontenc}

\usepackage{hyperref}       
\usepackage{url}            
\usepackage{booktabs}       
\usepackage{amsfonts,amsmath,amssymb,amsthm}       
\usepackage{nicefrac}       
\usepackage{microtype}      
\usepackage{lipsum}
\usepackage{xcolor}
\usepackage{graphicx,float}
\usepackage{tikz}
\usepackage[big]{layaureo}
\usepackage{perpage}
\MakePerPage{footnote}

\theoremstyle{plain}
\newtheorem{theorem}{Theorem}[section]
\newtheorem{lemma}{Lemma}[section]

\newtheorem{cor}{Corollary}[section]
\newtheorem{claim}{Claim}[section]
\newtheorem{thmout}{Theorem}

\theoremstyle{definition}
\theoremstyle{definition}
\newtheorem*{definition}{Definition}
\newtheorem*{remark}{Remark}

\definecolor{myred}{RGB}{226,56,18}
\definecolor{myorange}{RGB}{228,139,0}
\definecolor{mygreen}{RGB}{4,215,17}
\definecolor{mygrey}{RGB}{180,180,180}

\DeclareMathOperator{\Mod}{mod}

\def\Cx{\mathbb{C}}
\def\Chat{\widehat{\mathbb{C}}}

\def\dist{\mathrm{dist}}

\begin{document}

\title{\textbf{\textsc{multiply connected wandering domains of meromorphic functions: internal dynamics and connectivity}}}

\author{Gustavo R.~Ferreira\thanks{Email: \texttt{gustavo.rodrigues-ferreira@open.ac.uk}}\\
  \small{School of Mathematics and Statistics, The Open University}\\
  \small{Milton Keynes, MK7 6AA, UK}
}

\maketitle

\begin{abstract}
We discuss how the nine-way classification scheme devised by Benini \textit{et al.} for the dynamics of simply connected wandering domains of entire functions, based on the long-term behaviour of the hyperbolic distance between iterates of pairs of points and also the distance between orbits and the domains' boundaries, carries over to the general case of multiply connected wandering domains of meromorphic functions. Most strikingly, we see that not all pairs of points in such a wandering domain behave in the same way relative to the hyperbolic distance, and that the connectivity of the wandering domain greatly influences its possible internal dynamics. After illustrating our results with the well-studied case of Baker wandering domains, we further illustrate the diversity of multiply connected wandering domains in general by constructing a meromorphic function with a wandering domain without eventual connectivity. Finally, we show that an analogue of the ``convergence to the boundary'' classification of Benini \textit{et al.} does hold in general, and add new information about how this convergence takes place.

\end{abstract}


\section{Introduction}
Let $f:\Cx\to\Chat$ be a meromorphic function. In the study of its iterates, one sees the complex plane divided into two completely invariant subsets: the \textit{Fatou set}, where the family $\{f^n\}_{n\in\mathbb{N}}$ is normal, and its complement the \textit{Julia set}. The Fatou set, usually denoted $F(f)$, is -- by definition -- open, and its connected components (called Fatou components) are mapped into one another by $f$. This, in turn, separates Fatou components into two types: those that are eventually periodic, and whose behaviour was classified by Fatou himself a century ago \cite{Fat20}, and \textit{wandering domains} -- which have been the focus of increasingly intensive research since Baker first proved their existence in the 1970s \cite{Bak76}.

Wandering domains are astounding in their diversity: they can be simply or multiply connected, escaping or oscillating (the existence of wandering domains with only finite limit functions is still an open problem), bounded or unbounded, and they can even combine these characteristics along the same orbit. Regarding multiply connected ones (which are our goal here), Bergweiler, Rippon and Stallard \cite{BRS13} focused on the case of entire functions and obtained a detailed description of their geometric properties and internal dynamics: their iterates contain increasingly large ``absorbing'' annuli, which every orbit in the domain eventually enters.

However, if one considers meromorphic functions, not every multiply connected wandering domain can be treated with the same techniques used by Bergweiler, Rippon and Stallard. To deal with this diversity of wandering domains, we concentrate on a different aspect of the internal dynamics and the connectedness properties of the wandering domain.

In a recent paper, Benini \textit{et al.} \cite{BEFRS19} did exactly that: they devised a nine-way classification scheme for simply connected wandering domains based on the long-term behaviour of orbits relative to each other and relative to the boundary of the wandering domains. More specifically, let $z$ and $w$ be points in a simply connected wandering $U\subset F(f)$, and let $U_n$ denote the Fatou component containing $f^n(U)$ for any $n\in\mathbb{N}$. Then, one can consider the behaviour of the sequences
\[ \text{$d_{U_n}(f^n(z), f^n(w))$ and $\dist(f^n(z), \partial U_n)$}, \]
where $d_\Omega$ denotes hyperbolic distance in the domain $\Omega$; one of the most important findings of \cite{BEFRS19} is that \textit{all} points in $U$ behave in the same way concerning these sequences, which in turn allows us to neatly classify simply connected wandering domains according to such behaviour. Our primary aim for this paper, then, is to determine if their scheme is applicable to multiply connected wandering domains. Before stating our results, however, we would like to outline theirs. Regarding the hyperbolic metric (which is where we will see the biggest changes compared to the simply connected case) in particular, they proved the following.

\begin{thmout} \label{thmA}
Let $U$ be a simply connected wandering domain of a transcendental entire function $f$, and define the countable set of pairs
\[ E := \{(z, z')\in U\times U : f^k(z) = f^k(z')\text{ for some }k\in\mathbb{N}\}. \]
Then, exactly one of the following holds.
\begin{enumerate}[(i)]
    \item $d_{U_n}\left(f^n(z), f^n(w)\right)\to 0$ for all  $z, w\in U$, and we say that $U$ is \emph{contracting};
    \item $d_{U_n}\left(f^n(z), f^n(w)\right)\to c(z, w) > 0$ and $d_{U_n}\left(f^n(z), f^n(w)\right) > c(z, w)$ for all $z, w\in (U\times U)\setminus E$, and we say that $U$ is \emph{semi-contracting};
    \item There exists $N\in\mathbb{N}$ such that, for every pair $(z, w)\in (U\times U)\setminus E$, $d_{U_n}\left(f^n(z), f^n(w)\right) = c(z, w) > 0$ for $n\geq N$, and we say that $U$ is \emph{eventually isometric}.
\end{enumerate}
\end{thmout}

As concerns Theorem \ref{thmA}, we see that things are \textit{not} the same for multiply connected wandering domains: it is not true that all points in a wandering domain $U$ behave in the same way relative to $d_{U_n}$, and in fact even ``commonplace'' examples can combine different long-term behaviours of the hyperbolic metric. That is not to say, however, that we are lost at sea; even these ``mixed-type'' wandering domains can still exhibit interesting structures. We postpone a discussion of what this means until after Theorem \ref{thm:cutout} below; here, we make the following distinction.

\begin{definition}
Let $U$ be a multiply connected wandering domain of a transcendental meromorphic function $f$. We say that $U$ is \textit{trimodal} (respectively, \textit{bimodal}) if it exhibits all (resp. two out of three) possible behaviours described in Theorem \ref{thmA}.
\end{definition}

Another aspect of multiply connected wandering domains that helps us in classifying their internal dynamics is their connectivity $c(U)$, defined as the number of connected components of $\Chat\setminus U$ (which, of course, was a non-issue in the simply connected case). Recall that, as defined by Kisaka and Shishikura \cite{KS08}, the \textit{eventual connectivity} of a wandering domain $U$ (if it exists) is the number $k\in\mathbb{N}\cup\{+\infty\}$ such that $c(U_n) = k$ for all large $n$.  Notice that it often does exist -- Kisaka and Shishikura also showed that multiply connected wandering domains of entire functions always have an eventual connectivity, and it is either two or infinity. Our first result (proved in Section \ref{sec:hyp}) shows that a wandering domain's geometry and eventual connectivity severely restrict its possible internal dynamics.

\begin{theorem} \label{thm:cutout}
Let $U$ be a wandering domain of the meromorphic function $f$. Assume that $U$ has finite eventual connectivity $k$. Then,
\begin{enumerate}[(i)]
    \item if $k\geq 3$, $U$ is eventually isometric;
    \item if $k = 2$ and $\deg f|_{U_n}$ is finite for all large $n$, either
    \begin{enumerate}[(a)]
        \item $\Mod U_n$ is constant\footnote{For a definition of the modulus of $U_n$, see Subsection \ref{ssec:model}.} and $U$ is eventually isometric, or
        \item $\Mod U_n\to +\infty$ and $U$ is trimodal and admits contracting and eventually isometric transversal laminations.
    \end{enumerate}
\end{enumerate}
\end{theorem}

The laminations in the theorem above encode the aforementioned ``structure'' in mixed-type domains. A precise definition is the following.

\begin{definition}
Let $U$ be a wandering domain of the transcendental meromorphic function $f$.
\begin{itemize}
    \item We say that $U$ admits a \textit{contracting lamination} if there exists a lamination $\mathcal{C}$ of $U$ such that
    \[ \text{$d_{U_n}\left(f^n(z), f^n(w)\right)\to 0$ for all $z$ and $w$ on the same leaf of $\mathcal{C}$.} \]
    \item We say that $U$ admits an \textit{eventually isometric lamination} if there exists a lamination $\mathcal{L}$ of $U$ such that, for all large $n$,
    \[ \text{$d_{U_n}\left(f^n(z), f^n(w)\right) = c(z, w) > 0$ for every $z$ and $w$ on the same leaf of $\mathcal{L}$.} \]
\end{itemize}
\end{definition}

We will see that, if $U$ admits both a contracting \textit{and} an eventually isometric lamination, then points $z, w\in U$ that are not on the same leaf for either lamination behave ``semi-contractingly'' -- i.e.,
\[ \text{$d_{U_n}\left(f^n(z), f^n(w)\right)\searrow c(z, w) > 0$ as $n\to +\infty$.} \]
In other words, a wandering domain admitting both laminations is automatically trimodal (see Section \ref{sec:hyp} and Figure \ref{fig:laminations}).

Theorem \ref{thm:cutout} can be thought of as a ``silhouette theorem'': just by knowing the geometry of the iterates of a wandering domain, we can (in some cases) predict their internal dynamics. The hypothesis of finite eventual connectivity is necessary, since infinite connectivity offers, in general, sufficient flexibility for many kinds of behaviours (compare, for instance, \cite[Example 1]{RS08} and \cite[Theorem (iii)]{BKL90}). However, in a well-studied case, both finite and infinite connectivity are associated to particular internal dynamics.

To understand what this particular case is, we recall the definition of a \textit{Baker wandering domain}. A multiply connected wandering domain $U$ is said to be a Baker wandering domain if, for all large $n$, $U_n$ surrounds the origin and $U_n\to\infty$ as $n\to+\infty$ (see \cite{RS11} for a discussion of this and other closely related types of wandering domains). If the function in question is entire, every multiply connected wandering domain is of this kind, and this remains true if we allow the function to have finitely many poles (see \cite{RS08}). An even larger class of functions with Baker wandering domains with similar properties is the class of transcendental meromorphic functions with a direct tract (see Section \ref{sec:dir} for a definition). As we will see in Section \ref{sec:dir}, all Baker wandering domains in this larger class have asymptotic behaviour similar to that of Baker wandering domains of entire functions, and -- as our next ``silhouette theorem'' shows -- have very rigid internal dynamics.

\begin{theorem} \label{thm:Baker}
Let $f$ be a transcendental meromorphic function with a direct tract $D$ and a Baker wandering domain $U$. Then,
\begin{enumerate}[(i)]
    \item $U$ has eventual connectivity either two or infinity;
    \item $U$ admits a contracting lamination $\mathcal{C}$ made of level sets of a harmonic function, and
    \begin{enumerate}[(a)]
        \item if $c(U) < +\infty$, then $U$ is trimodal. More specifically, it also admits an eventually isometric lamination $\mathcal{L}$ transversal to $\mathcal{C}$.
        \item if $c(U) = +\infty$, then $U$ is bimodal. More specifically, points on different leaves of $\mathcal{C}$ behave ``semi-contractingly''.
    \end{enumerate}
\end{enumerate}
\end{theorem}

Now, an interesting question regarding Theorem \ref{thm:cutout} is whether the hypothesis that $U$ \textit{has} an eventual connectivity is necessary; every previously known example of a multiply connected wandering domain does. However, it turns out that this is not always the case: In Section \ref{sec:ex}, we will invoke Arakelyan's theorem to obtain a sequence of meromorphic functions approximating hand-picked functions, and will use it to construct the following example.

\begin{theorem} \label{thm:ex}
There exists a transcendental meromorphic function $g$ with a wandering domain $U$ such that
\begin{enumerate}[(i)]
    \item each $U_{4k}$, $k\geq 0$, is unbounded and simply connected,
    \item each $U_{4k+1}$, $k\geq 0$, is bounded and doubly connected,
    \item each $U_{4k+2}$, $k\geq 0$, is bounded and simply connected, and
    \item each $U_{4k+3}$, $k\geq 0$, is unbounded and simply connected.
\end{enumerate}
\end{theorem}

After such strange and wild behaviour from multiply connected wandering domains, it is encouraging to know that not everything about Benini \textit{et al.}'s classification scheme is overturned. As our final theorem (proved in Section \ref{sec:boundary}) shows, it remains (mostly) true that all orbits behave the same way regarding convergence to the boundary. However, we need to restrict our attention to a particular part of the boundary -- the ``outer'' boundary, defined for our purposes as the boundary of $\widetilde{U}$, the \textit{topological convex hull} of $U$, which is the union of $U$ and its bounded complementary components. Notice that $\partial\widetilde{U}$ is not necessarily connected or bounded; take for example the outer boundary of $\{z : |\Im z| < 1\}$.

\begin{theorem} \label{thm:classII}
Let $U$ be a wandering domain of a transcendental meromorphic function $f$. Then, exactly one of the following holds.
\begin{enumerate}[(a)]
    \item $\liminf_{n\to+\infty} \dist\left(f^n(z), \partial \widetilde{U}_n\right) > 0$ for all $z\in U$;
    \item there exists a subsequence $n_k\to+\infty$ for which $\dist\left(f^{n_k}(z), \partial \widetilde{U}_{n_k}\right)\to 0$ for all $z\in U$, and a different subsequence $m_k\to +\infty$ for which $\liminf_{k\to+\infty} \dist\left(f^{m_k}(z), \partial \widetilde{U}_{m_k}\right) > 0$ for all $z\in U$;
    \item $\dist\left(f^n(z), \partial\widetilde{U}_n\right)\to 0$ for all $z\in U$.
\end{enumerate}
Additionally, in cases (b) and (c), let $w_k\in\partial \widetilde{U}_{n_k}$ be such that $\dist\left(f^{n_k}(z), \partial\widetilde{U}_{n_k}\right) = |f^{n_k}(z) - w_k|$. Then, for every other $z'\in U$, we have $|f^{n_k}(z') - w_k| \to 0$ as $k\to +\infty$.
\end{theorem}

The final assertion saying that all orbits converge to ``the same parts'' of the boundary (when they do so at all) is new even for simply connected wandering domains, but still uses the original techniques of \cite{BEFRS19}.

\textsc{Acknowledgements.} I would like to thank my supervisors, Phil Rippon and Gwyneth Stallard, for their encouragement, comments, and suggestions about this work.

\section{The hyperbolic metric and multiply connected wandering domains} \label{sec:hyp}
This section is devoted to the proof of Theorem \ref{thm:cutout}. It is divided in two parts: first, we study a ``toy'' model of composing power maps between annuli; then, we justify the attention given to such a simple model by showing that it is in many ways equivalent to wandering domains of eventual connectivity two (the proof of Theorem \ref{thm:cutout}(i) is simple, and we do not dwell extensively on it).

Before that, however, we must clarify what we mean by a lamination, since this concept is central to transferring our annulus-based knowledge to a general setting. We would like to point out that different texts use slightly different definitions; ours is in the spirit of \cite{KH95} and \cite{Lee13}, and is tailored to the kinds of manifolds we will meet here.

\begin{definition}
Let $X$ be a Riemann surface, and let $Y\subseteq X$ be a subset such that $X\setminus Y$ is at most countable. A \textit{lamination} of $X$ is a partition $\{L_\alpha\}_{\alpha\in A}$ of $Y$ into injectively immersed real submanifolds of (real) dimension one such that:
\begin{itemize}
    \item $L_\alpha\cap L_\beta = \emptyset$ whenever $\alpha\neq\beta$;
    \item For every $p\in Y$, there exists a neighbourhood $U$ of $p$ and a conformal isomorphism $h:U\to h(U)\subset\Cx$ such that $h(L_\alpha\cap U)$ is either empty or of the form $\{z\in h(U) : \Im z = k\}$ for some constant $k = k(\alpha)$ (i.e., $L_\alpha$ is ``straightened'' onto a line segment, and different $\alpha$'s lead to different parallel line segments).
\end{itemize}
The submanifolds $\{L_\alpha\}_{\alpha\in A}$ are called the \textit{leaves} of the lamination. If $Y = X$, the lamination is called a \textit{foliation} of $X$.
\end{definition}

\subsection{The annulus model} \label{ssec:model}
Now, let us consider the composition of power mappings between annuli; first, we must understand our domain. For any $R > 1$, we define
\[ A(R) := \{z\in\Cx : 1/R < |z| < R\}; \]
the modulus of $A(R)$ is given by\footnote{We're using Beardon and Minda's definition \cite{BM06}; other authors normalise it by a factor of $2\pi$.}
\[ \Mod A(R) := \log \frac{R}{R^{-1}} = 2\log R. \]
It is well known that every doubly connected domain on $\Chat$ is conformally isomorphic to either $\Cx^*$, $\mathbb{D}^*$, or $A(R)$ for some $R > 1$, that these model spaces are all incompatible with each other, and that $\Mod A(R)$ is a conformal invariant defining equivalence classes of doubly connected domains with non-degenerate complementary components (see, for instance, \cite[Section 6.5]{Ahl79}). With that in mind, let us take a closer look at $A(R)$; particularly important subsets are the circles
\[ C_r := \{z\in A(R) : |z| = r\}, r\in (1/R, R), \] and the ray segments
\[ \text{$L_\theta := \{z\in A(R) : \arg z = \theta\}$ for $\theta\in [0, 2\pi)$.} \]
In the following two lemmas, we gather some facts about the hyperbolic metric in $A(R)$; these facts are either ``clear'' from an explicit universal covering of the annulus, or can be found in \cite{BM06}, \cite{Com11}, or \cite[Chapter 1]{Bus10}.

\begin{lemma} \label{lem:geods}
For any $R > 1$,
\begin{enumerate}[(i)]
    \item For any $\theta\in[0, 2\pi)$, $L_\theta$ is a hyperbolic geodesic of $A(R)$. Geodesics arcs in $L_\theta$ are distance-minimising for any two points on the same $L_\theta$, and are also unique in their homotopy class.
    \item The parametrisation $\gamma(t) = \exp(2\pi int)$, $t\in[0, 1]$, of $C_1$ traversed $n\in\mathbb{Z}\setminus\{0\}$ times is the only closed geodesic in its free homotopy class in $A(R)$. Its length is given by
    \[ \ell_{A(R)}(\gamma) = \frac{2\pi^2|n|}{\Mod A(R)}. \]
\end{enumerate}
\end{lemma}
\begin{lemma} \label{lem:coverings}
Let $R, S > 1$, let $f:A(R)\to A(S)$ be a holomorphic mapping, and define $\gamma(t) = \exp(2\pi it)$  for $t\in[0, 1]$. If $n = I(f\circ\gamma, 0)$, then
\[ |n| \leq \frac{\Mod A(S)}{\Mod A(R)}, \]
with equality if and only if $S = R^{|n|}$ and $f(z) = e^{i\theta}z^n$ for some $\theta\in\mathbb{R}$.
\end{lemma}

With these in place, let us begin. We fix some $R > 1$, take a sequence $(d_n)_{n\in\mathbb{N}}$ of natural numbers, and define $D_0 = 1$ and, for $n\geq 1$,
\[ D_n := \prod_{k=1}^n d_n. \]
We assume for the sake of convenience that $D_n\to +\infty$ as $n\to +\infty$; otherwise, we would have $d_n = 1$ for all large $n$, which does not lead to interesting behaviour. Now, take the sequence $f_n:A(R^{D_{n-1}})\to A(R^{D_n})$, $n\geq 1$, given by $f_n(z) = z^{d_n}$; it is clear that we can compose these maps, obtaining
\[ F_n = f_n\circ f_{n-1}\circ\cdots\circ f_1 : A(R)\to A(R^{D_n}). \]
Without further ado, let us describe the long-term behaviour of the hyperbolic metric relative to $F_n$.

\begin{theorem} \label{thm:ann}
For any choice of $R > 1$ and sequence $(d_n)_{n\in\mathbb{N}}$ as above, any pair $z, w\in A(R)$ satisfies exactly one of the following:
\begin{enumerate}[(i)]
    \item $z$ and $w$ belong to the same $C_r$ for some $r\in A(1/R, R)$, and $d_{A(R^{D_n})}\left(F_n(z), F_n(w)\right)\to 0$ as $n\to+\infty$;
    \item $z$ and $w$ belong to the same $L_\theta$ for some $\theta\in[0, 2\pi)$, and $d_{A(R^{D_n})}\left(F_n(z), F_n(w)\right) = d_{A(R)}(z, w)$ for all $n\in\mathbb{N}$;
    \item $z$ and $w$ are neither on the same circle nor on the same ray, and $d_{A(R^{D_n})}\left(F_n(z), F_n(w)\right)\to c(z, w) > 0$ without ever being equal to $c(z, w)$.
\end{enumerate}
\end{theorem}

It is easy to see that the collections $\mathcal{C} := \{C_r : 1/R < r < R\}$ and $\mathcal{L} := \{L_\theta : 0\leq \theta < 2\pi\}$ form two transversal foliations of $A(R)$; $\mathcal{C}$ is contracting, $\mathcal{L}$ is isometric, and together they encode all relevant information regarding the composition of the power maps $f_n$ (see Figure \ref{fig:laminations}). Now, let us prove Theorem \ref{thm:ann}.

\begin{proof}[Proof of Theorem \ref{thm:ann}]
The first part of our claim is clear; any pair $z, w\in A(R)$ is either on the same circle, on the same ray, or neither. We are left to associate each case to its dynamical consequences, which we'll do on a case-by-case basis.

First, assume that $z$ and $w$ are on the same circle $C_r\subset A(R)$, which we parametrise as $\gamma(t) = r\exp(2\pi it)$, $t\in[0, 1]$. Since $F_n:A(R)\to A(R^{D_n})$ is locally isometric, we have $\ell_{A(R)}(\gamma) = \ell_{A(R^{D_n})}(F_n\circ\gamma)$, where (here and throughout) $\ell_\Omega$ denotes the hyperbolic length of a curve; we also know that $F_n(z) = z^{D_n}$, so that $F_n\circ\gamma(t) = r^{D_n}\exp(2\pi iD_nt)$. Thus, $F_n\circ \gamma$ is a parametrisation of the circle $C_{r^{D_n}}\subset A(R^{D_n})$, traversed $D_n$ times with constant speed. It follows that the length $\ell_{A(R^{D_n})}(C_{r^{D_n}})$ of $C_{r^{D_n}}$ traversed only once is $\ell_{A(R)}(C_r)/D_n$; connecting $F_n(z)$ to $F_n(w)$ with a simple arc on $C_{r^{D_n}}$, we see that
\[ d_{A(R^{D_n})}\left(F_n(z), F_n(w)\right)\leq \ell_{A(R^{D_n})}(C_{r^{D_n}}) = \frac{\ell_{A(R)}(C_r)}{D_n}, \]
and it is clear that the right-hand side goes to zero as $n\to+\infty$.

Now, assume that $|z| \neq |w|$, but $z$ and $w$ are both on the same ray $L_\theta\subset A(R)$. We take a distance-minimising geodesic $\gamma\subset A(R)$ connecting $z$ to $w$, and we know by Lemma \ref{lem:geods}(i) that $\gamma\subset L_\theta$. Again by the fact that $F_n$ is locally isometric, we have $\ell_{A(R^{D_n})}(F_n\circ\gamma) = \ell_{A(R)}(\gamma) = d_{A(R)}(z, w)$. Unlike in the previous case, however, we know that $F_n$ actually takes $L_\theta$ to $L_{D_n\theta}$ in one-to-one fashion, meaning that $F_n\circ\gamma$ is a distance-minimising geodesic arc in $L_{D_n\theta}$ and $\ell_{A(R^{D_n})}(F_n\circ\gamma) = d_{A(R^{D_n})}\left(F_n(z), F_n(w)\right)$.

Finally, assume that $z$ and $w$ are neither on the same circle nor on the same ray. There exists, however, some $z^*\in A(R)$ such that $|z^*| = |z|$ and $\arg z^* = \arg w$, meaning that $z^*$ belongs to the same circle as $z$ and to the same ray as $w$. By the reverse triangle inequality, we have
\[ d_{A(R^{D_n})}\left(F_n(z), F_n(w)\right) \geq \left|d_{A(R^{D_n})}\left(F_n(z), F_n(z^*)\right) - d_{A(R^{D_n})}\left(F_n(z^*), F_n(w)\right)\right|; \]
items (i) and (ii) of this theorem now tell us that $d_{A(R^{D_n})}\left(F_n(z), F_n(z^*)\right)\to 0$ as $n\to +\infty$, while $d_{A(R^{D_n})}\left(F_n(z^*), F_n(w)\right) = d_{A(R)}(z^*, w)$ for all $n$. This yields a positive lower bound for $c(z, w) = \lim_{n\to+\infty} d_{A(R^{D_n})}\left(F_n(z), F_n(w)\right)$, which exists since $d_{A(R^{D_n})}\left(F_n(z), F_n(w)\right)$ is a decreasing sequence bounded from below.
\end{proof}

\begin{figure}[!h]
    \centering
    \includegraphics[width=0.6\columnwidth]{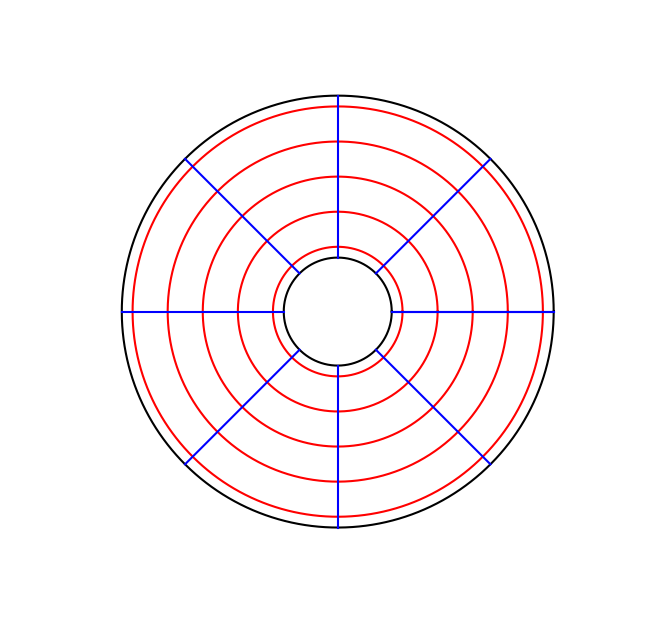}
    \caption{An annulus $A(R)$ with its contracting foliation $\mathcal{C}$ (in red) and isometric foliation $\mathcal{L}$ (in blue).}
    \label{fig:laminations}
\end{figure}

\subsection{Finite eventual connectivity}
Now, we consider a transcendental meromorphic function $f:\Cx\to\Chat$ with a wandering domain $U$, which we will assume has finite eventual connectivity $k$. First, in order to deal with case (ii) of Theorem \ref{thm:cutout}, we want to show that, in the case $k = 2$, we can ``conjugate'' the dynamics of $f|_{U_n}$ to the model discussed in Subsection \ref{ssec:model}. For that, we will need the following amalgamation of Theorems 1 and 3 from \cite{Bol99}, proved using Ahlfors's theory of covering surfaces.

\begin{lemma} \label{lem:bolsch}
Let $f$ be a transcendental meromorphic function. If $V\subset \Chat$ is an arbitrary domain and $U$ is any component of $f^{-1}(V)$, then exactly one of the following holds.
\begin{enumerate}[(i)]
    \item There exists $n\in\mathbb{N}$ such that $c(U) - 2 = n\cdot\left(c(V) - 2\right) + \delta_f(U)$ and $n = \deg(f|_U)$;
    \item $f(U)$ covers every value of $V$ infinitely often with at most two exceptions, and $c(V)\geq 3$ implies $c(U) = +\infty$.
\end{enumerate}
\end{lemma}

With this in mind, let us begin.

\begin{lemma} \label{lem:conj}
Let $f:\Cx\to\Chat$ be a meromorphic function with a wandering domain $U$. Assume that $c(U_n) = 2$ and $\deg f|_{U_{n-1}} = d_n < +\infty$ for all $n\geq 0$. Then, there exists $R > 1$ and a sequence of conformal isomorphisms $\varphi_n:U_n\to A(R^{D_n})$ such that, for $n\geq 1$,
\[ \text{$\varphi_n\circ f(z) = \left(\varphi_{n-1}(z)\right)^{d_n}$ for $z\in U_{n-1}$.} \]
\end{lemma}
\begin{proof}
We proceed inductively; let us start with $n = 1$. Then, $U = U_0$ is conformally isomorphic to a unique symmetric annulus $A(R)$, and the isomorphism $\varphi_0:U_0\to A(R)$ is unique up to rotation and inversion. The next iterate $U_1$ is also isomorphic to some unique $A(S)$, but in this case we might have to post-compose the isomorphism $\varphi_1:U_1\to A(S)$ with a rotation or inversion. For now, this induces a holomorphic map $g_1:A(R)\to A(S)$ by
\[ g_1(z) = \varphi_1\circ f\circ \varphi_0^{-1}(z), \]
and we know that $\deg g_1 = \deg f|_{U_0} = d_1$ is finite -- we exercise our freedom of choice over $\varphi_1$ here, choosing it so that $d_1$ is also positive. Thus, by Lemma \ref{lem:coverings},
\[ d_1\leq \frac{\Mod A(S)}{\Mod A(R)}; \]
we want to show that equality holds, whence (also by Lemma \ref{lem:coverings}) $g_1(z) = e^{i\theta}z^{d_1}$, whereupon we exercise our freedom of choice again to rotate $\varphi_1$ and ensure that $e^{i\theta} = 1$. To this end, notice that, by the Riemann-Hurwitz formula (Lemma \ref{lem:bolsch}(i)), $f|_{U_0}$ (and hence $g_1$) has no critical points, and is therefore an unbranched covering map. This means that $g_1:A(R)\to A(S)$ is a local hyperbolic isometry; it takes the closed geodesic $\gamma(t) = \exp(2\pi it)$, $t\in[0, 1]$, to a closed geodesic of $A(S)$. By Lemma \ref{lem:geods}(ii), $g_1\circ\gamma$ must be another (monotonic) parametrisation of the unit circle -- but one that traverses it $d_1$ times. Since $\ell_{A(R)}(\gamma) = \ell_{A(S)}(g_1\circ\gamma)$, Lemma \ref{lem:geods}(ii) also tells us that
\[ \frac{2\pi^2}{\Mod A(R)} = \frac{2\pi^2d_1}{\Mod A(S)}, \]
whence $\Mod A(S) = d_1\cdot\Mod A(R)$ and we are done. The rest of the sequences $(\varphi_n)$ and $(g_n)$ can be built in a similar fashion, with $\varphi_{n+1}$ being rotated and inverted as necessary to accommodate $\varphi_n$ (which is already fixed at the $n$-th stage).
\end{proof}

An immediate consequence of Lemma \ref{lem:conj} is that, by induction,
\[ \text{$\varphi_n\circ f^n = G_n\circ\varphi_0$ for all $n\geq 1$,} \]
where $G_n := g_n\circ\cdots\circ g_0$, each $g_n$ is a power map, and the $\varphi_n$ are hyperbolic isometries. If $\mathcal{C}''$ and $\mathcal{L}''$ are the foliations given by Theorem \ref{thm:ann} for $G_n$, we can pull them back to obtain transversal foliations $\mathcal{C}' := (\varphi_0)^*\mathcal{C}'' = \{\varphi_0^{-1}(C) : C\in\mathcal{C}\}$ and $\mathcal{L}' := (\varphi_0)^*\mathcal{L}'' = \{\varphi_0^{-1}(L) : L\in\mathcal{L}\}$ on $U$, and -- since all $\varphi_n$ are hyperbolic isometries -- they encode the same dynamical information for $U$ and $f$ that we had for $A(R)$ and $G_n$.

This takes care of the case of \textit{constant} connectivity two; next, we complete the proof of Theorem \ref{thm:cutout}.

\begin{proof}[Proof of Theorem \ref{thm:cutout}]
Assume first that $U$ has eventual connectivity $3\leq k < +\infty$. We claim that $f:U_n\to U_{n+1}$ is a proper map for all large $n$; indeed, the only way for it \textit{not} to be proper is for $m = \deg f|_{U_n}$ to be infinite (since $f$ maps $\partial U_n$ onto $\partial U_{n+1}$, it follows from Ahlfors's first fundamental theorem that $f:U_n\to U_{n+1}$ either is proper and satisfies the Riemann-Hurwitz formula, or has infinite degree; see for example \cite[Theorem 5.2]{Hay64}, \cite[Theorem VI.1]{Tsu75}, or \cite[Theorems 2 and 3]{Bol99}). However, by Lemma \ref{lem:bolsch}(ii), since $c(U_{n+1}) = k\geq 3$ we should have $c(U_n) = +\infty$, which is a contradiction as $c(U_n) = k$.

Thus, we can apply the Riemann-Hurwitz formula to $f:U_n\to U_{n+1}$, which tells us that
\[ k - 2 = m\cdot(k - 2) + \delta_f(U_n). \]
Since $m$ and $\delta_f(U_n)$ are both non-negative integers and $k\geq 3$, the only solution is $\delta_f(U_n) = 0$ and $m = 1$. It follows that $f:U_n\to U_{n+1}$ is a conformal isomorphism for all large $n$, and thus that the hyperbolic metric is preserved by $f|_{U_n}$.

For $k = 2$, we take a sufficiently large $N$ that $c(U_n) = 2$ for all $n\geq N$, and apply Lemma \ref{lem:conj}. If $\Mod U_n$ is a constant sequence, then $f|_{U_n}$ is conjugated to rigid rotations of the same $A(R)$ for some $R > 1$, and so the hyperbolic metric is preserved. If $\Mod U_n$ is not constant, it is (by Lemma \ref{lem:conj}) an increasing sequence diverging to infinity, and we obtain transversal foliations $\mathcal{C}'$ and $\mathcal{L}'$ on $U_N$ that encode its dynamics. In order to transfer this knowledge to $U$, we pull the foliations back as $\mathcal{C} := (f^N)^*\mathcal{C}' = \{f^{-N}(C)\cap U : C\in\mathcal{C}'\}$ and $\mathcal{L} := (f^N)^*\mathcal{L}' = \{f^{-N}(L)\cap U : L\in\mathcal{L}'\}$. This process ``breaks down'' at the critical points $\mathrm{Crit}(f^N|_U)$ of $f^N$, which form a discrete set, but works conformally everywhere else. It follows that $\mathcal{C}$ and $\mathcal{L}$ are transversal laminations of $U$ that fail to be foliations at critical points of $f^N$. By their definitions and Theorem \ref{thm:ann}, it also follows that $\mathcal{C}$ is a contracting lamination, and $\mathcal{L}$ is an eventually isometric one.
\end{proof}

\section{Baker wandering domains} \label{sec:dir}
In this section, we will examine the internal dynamics of Baker wandering domains in direct tracts. Of course, in order to do so, we should first define the latter: an unbounded domain $D\subset \Cx$ with piecewise smooth boundary is said to be a \textit{direct tract} for the meromorphic function $f$ if $\Cx\setminus D$ is unbounded, $f$ has no poles in $D$, and there exists $R > 0$ such that $|f(z)| = R$ for $z\in\partial D$ while $|f(z)| > R$ for $z\in D$. Every meromorphic function with finitely many poles has a direct tract, and so do many with infinitely many poles (such as Euler's gamma function, for instance); this gives us a substantially larger class to study than entire functions, while still leaving us with plenty of machinery to do so.

Of course, when talking about the dynamics of Baker wandering domains, one must talk about the results of Bergweiler, Rippon, and Stallard \cite{BRS13}. Although they deal with entire functions, they remark that their results can be generalised to Baker wandering domains of meromorphic functions with direct tracts. The one key step not directly related to their techniques is to generalise a theorem -- originally proved by Zheng \cite{Zhe01} for functions with finitely many poles -- saying that the iterates of a Baker wandering domain contain large annuli $\{z : r_n < |z| < R_n\}$ with $R_n/r_n\to +\infty$. We now give a brief outline of how to do this.

Many of the results and tools used here were also introduced by Bergweiler, Rippon, and Stallard in a different paper \cite{BRS08}. Firstly, they showed that if $U$ is a Baker wandering domain of $f$ and $f$ has a direct tract $D$, then $\overline{U}_n\subset D$ for all large $n$, so that $D$ is in fact the only direct tract of $f$ and all components of $\Cx\setminus D$ are bounded. Secondly, they introduced the subharmonic function $v:\Cx\to[0, +\infty)$, defined as
\[ v(z) := \begin{cases} \log\frac{|f(z)|}{R}, & z\in D, \\ 0, & z\notin D, \end{cases} \]
where $D$ is a direct tract of $f$. Combined, these two things have key consequences for the distribution of zeros and poles of $f$.

\begin{lemma} \label{lem:ZP}
Let $f$ be a transcendental meromorphic function with a direct tract $D$ and a Baker wandering domain $U$. For $n\in\mathbb{N}$, let $\gamma_n\subset U_n$ be a Jordan curve surrounding the origin, and define $Z_n$ and $P_n$ as the number of zeros and poles of $f$ (respectively) surrounded by $\gamma_n$. Then,
 \[ \text{$Z_n - P_n\to +\infty$ as $n\to +\infty$.} \]
\end{lemma}
\begin{proof}
First, notice that $f$ has neither zeros nor poles in $D$ by definition, which means that our problem reduces to counting zeros and poles for the complementary components of $D$, which are all bounded since $f$ has a Baker wandering domain.
Now, let $\mu_v$ denote the Riesz measure associated to the subharmonic function $v$ (see, for instance, \cite[Section 3.5]{Hay76}). If $K$ is a complementary component of $D$, then \cite[Lemma 9.2]{BRS08} tells us that $\mu_v(K)\geq 1$, and it follows from the argument principle that $\mu_v(K) = Z_K - P_K$, where $Z_K$ and $P_K$ are the number of zeros and poles of $f$ in $\mathrm{int}(K)$, respectively (see the proof of \cite[Lemma 9.1]{BRS08}).

Next, let $K_1, K_2, \ldots, K_m$ denote the complementary components of $D$ surrounded by $\gamma_n$; clearly, $m$ is finite and depends on $n$. Then, it follows that
\[ Z_n - P_n = \sum_{i=1}^m (Z_{K_i} - P_{K_i})\geq m; \]
since all complementary components of $D$ are bounded and $U_n = f^n(U)\to\infty$ as $n\to +\infty$, we also have $m\to +\infty$ and the conclusion follows.
\end{proof}

The function $v$ was also used by Bergweiler, Rippon, and Stallard to prove an analogue of the Wiman-Valiron theorem for meromorphic functions with direct tracts (see \cite[Theorems 2.2 and 2.3]{BRS08}). This, in turn, implies that if $f$ is a meromorphic function with a direct tract $D$, then the image of any sufficiently large annulus lying sufficiently far away from the origin contains another large annulus lying even farther away. This fact coupled with Lemma \ref{lem:ZP} shows that, as concerns Baker wandering domains, functions with a direct tract behave similarly to functions with finitely many poles.

As promised, Lemma \ref{lem:ZP} and the generalised Wiman-Valiron theorem (the latter to substitute for Bohr's theorem) allow us to extend Zheng's theorem to functions with a direct tract. This, in turn, serves as the starting point to generalise Bergweiler, Rippon, and Stallard's results \cite{BRS13} concerning the harmonic function
\begin{equation} \label{eq:h}
    \text{$h(z) := \lim_{n\to+\infty} \frac{\log|f^n(z)|}{\log|f^n(z_0)|}$ for $z\in U$,}
\end{equation}
where $z_0$ is any point in a Baker wandering domain $U$ of an entire function, and so to generalise their approach to describing the dynamics of Baker wandering domains -- as was stated above; we refer to \cite{BRS13} and \cite{BRS08} for the remaining details. Thus, given a transcendental meromorphic function $f$ with a Baker wandering domain $U$ and a direct tract $D$, and a point $z_0\in U$, Equation (\ref{eq:h}) defines a positive, non-constant harmonic function, and the properties of $h$ tell us many things about the internal dynamics of $U$.

Naturally, the reason we did this was so we could apply the results in \cite{BRS13} to our setting. However, we will have to change our notation slightly; since $h$ can be defined taking as starting points \textit{any} point in \textit{any} domain on the orbit, we will use $h(z; z_0, U)$ to mean the function defined in $U$ by Equation (\ref{eq:h}) using the base point $z_0\in U$. With that in mind, it is clear that $h$ is $f$-invariant in the sense that
\begin{equation} \label{eq:finv}
 h\left(f(z); f(z_0), U_1\right) = h(z; z_0, U);
\end{equation}
notice that, since every Baker wandering domain is bounded and $f$ preserves the Fatou and Julia sets, $f(U) = U_1$ in this case.

The level sets of $h$ form a lamination of $U$, and this lamination was already studied by Sixsmith \cite{Six13} in relation to the fast escaping set of entire functions. Here, we are interested in level sets of $h$ for another reason.

\begin{lemma} \label{lem:h}
Let $f$ be a transcendental meromorphic function with a direct tract $D$ and a Baker wandering domain $U$. Choose $z_0\in U$, and define $h:U\to(0, +\infty)$ according to Equation (\ref{eq:h}). Then, the level sets of $h$ form a contracting lamination of $U$.
\end{lemma}
Before proving Lemma \ref{lem:h}, we want to convince ourselves that the level curves of $h$ are ``cilivised''. Of course, being level sets of a harmonic function, we know that they are made of analytic curves, but we want more than that.
\begin{lemma} \label{lem:closed}
In the setting of Lemma \ref{lem:h}, every level curve of $h$ is closed.
\end{lemma}
\begin{proof}
Assume that this is not the case; that is, that $h$ has a level curve $\gamma$ with $h(\gamma) = L\in\mathbb{R}_+$ that is not closed. Since $h$ is harmonic, $\gamma$ must escape to the boundary of $U$, and by \cite[Theorem 1.6(b)]{BRS13} (which tells us that $h$ has a continuous extension to $\partial U\setminus\partial\widetilde{U}$, and is constant there) and the maximum principle it must escape to the outer boundary\footnote{Bergweiler, Rippon, and Stallard use a different definition of the outer boundary, but -- fortunately -- it is equivalent to ours on bounded domains.} of $U$. Note that, since every $U_n$ is bounded and the Julia set is completely invariant, $f^n:U\to U_n$ is a proper map for every $n$ (see \cite[Lemma 4]{RS08}), and therefore $f^n(\gamma)$ will always be a level curve of $h(f^n(\cdot); f^n(z_0), U_n)$ that escapes to the outer boundary of $U_n$.

Take now another level set $\Gamma$ of $h$ for which $h(\Gamma) = L' > \max\{L, 1\}$ (by the maximum principle, such a level set must be non-empty for an appropriate choice of $L'$), and any point $z\in \gamma$. The definition of $h$ implies that $|f^n(z_0)|^{L - \epsilon_n} < |f^n(z)| < |f^n(z_0)|^{L + \epsilon_n}$, where $\epsilon_n\searrow 0$; in other words, $f^n(z)$ lies in some definite annulus $A_n$. At the same time, \cite[Theorem 7.1]{BRS13} says that $f^n(\Gamma)$ has (for large enough $n$) a connected component $\Gamma_n\subset U_n$ that is a Jordan curve surrounding the origin and lying in the annulus
\[ A_n' := \{z : |f^n(z_0)|^{L'(1 - \epsilon_n')} \leq |z| \leq |f^n(z_0)|^{L'(1 + \epsilon_n'')}\}, \]
where $\epsilon_n'$ and $\epsilon_n''$ are positive sequences going to zero with $n\to+\infty$. Most importantly, we see that if $n$ is large enough the annuli $A_n$ and $A_n'$ are disjoint, with $A_n'$ surrounding $A_n$.

Finally, recall that $f^n(\gamma)$ was supposed to escape to the outer boundary of $U_n$. It follows from the Jordan curve theorem that $f^n(\gamma)$ intercepts $\Gamma_n$, which is a contradiction since both are supposed to be level curves of $h(f^n(\cdot); f^n(z_0), U_n)$ corresponding to different levels (by Equation (\ref{eq:finv})).
\end{proof}
Thus pacified, we can prove Lemma \ref{lem:h}.
\begin{proof}[Proof of Lemma~\ref{lem:h}]
Let $\gamma\subset U$ be a simple closed level curve of $h$, and define $\gamma_n = f^n\circ\gamma\subset U_n$. Notice that $\gamma_n$ is another closed level curve of $h$ by Equation (\ref{eq:finv}), and it is clear from the maximum principle that every closed subcurve of $\gamma_n$ surrounds at least one bounded component of $\Cx\setminus U_n$. If $n$ is sufficiently large, then $\gamma_n$ is (by \cite[Theorem 1.3]{BRS13}) contained in a large annulus $C_n\subset U_n$ centred at $0$, with its iterates $f^m\circ\gamma_n$ being contained in similar annuli $C_{m+n}\subset U_{m+n}$. It follows that, for such $n$, $\gamma_n$ is a simple closed curve, which we will take to be traversed once. Assume now that $U$ has eventual connectivity two (we will postpone proving that $U$ has eventual connectivity either two or infinity to the end of this section). Then (assuming that $n$ is large enough), $\gamma_n$ surrounds the \textit{only} bounded complementary component of $U_n$, and by the argument principle $f\circ\gamma_n$ winds around the origin $d_n = Z_n - P_n$ times (we are using the notation of Lemma \ref{lem:ZP}). We see that $f\circ\gamma_n$ is a simple closed curve traversed $d_n$ times; let $\gamma_{n+1}$ stand for the same curve traversed only once. Since $\ell_{U_{n+1}}(f\circ\gamma_n) \leq \ell_{U_n}(\gamma_n)$ by the Schwarz-Pick lemma, we have that
\[ \ell_{U_{n+1}}(\gamma_{n+1}) \leq \frac{\ell_{U_n}(\gamma_n)}{d_n}, \]
and by induction and the $f$-invariance of $h$ we conclude that
\[ \text{$\ell_{U_{m+n}}(\gamma_{n+m}) \to 0$ as $m\to+\infty$.} \]
Therefore, if $z$ and $w$ are any two points of $\gamma_n$, we clearly have
\[ \text{$d_{U_{m+n}}\left(f^m(z), f^m(w)\right) \leq \ell_{U_{m+n}}(\gamma_{m+n})\to 0$ as $m\to +\infty$.} \]

This concludes the proof for eventual connectivity two; if $c(U) = +\infty$, we must deal with the simple closed level curves $\gamma\subset U$ of $h$ that do not surround the component of $\Cx\setminus U$ containing the origin (by Lemma \ref{lem:closed}, we have no non-closed level curves to deal with). However, by \cite[Lemma 2]{RS08}, any closed curve $\gamma\subset U$ that is not null-homotopic in $U$ must have an iterate $f^n\circ\gamma$ that surrounds a pole of $f$, and therefore a bounded complementary component of $D$. In that case, by Lemma \ref{lem:ZP} and the argument principle, $f^{n+1}\circ\gamma$ surrounds the origin, and the previous arguments apply.
\end{proof}

Finally, we complete the proof of Theorem \ref{thm:Baker}.
\begin{proof}[Proof of Theorem \ref{thm:Baker}]
First, we establish the existence of an eventual connectivity of $U$; this will require only a small adaptation of Kisaka and Shishikura's argument for \cite[Theorem A]{KS08}.

Since every $U_n$ is bounded, the Riemann-Hurwitz formula tells us that $c(U_n)$ is a non-increasing sequence, which means that $U$ has a well-defined eventual connectivity. Now, if this connectivity is some finite $k\geq 3$, Theorem \ref{thm:cutout} implies that $f:U_n\to U_{n+1}$ is one-to-one for all large $n$. Hence, if $\gamma_n\subset U_n$ is a Jordan curve surrounding the origin, the argument principle tells us that (with the notation of Lemma \ref{lem:ZP}) $Z_n - P_n = 1$ for every large $n$, directly contradicting Lemma \ref{lem:ZP}.

Now, we associate the eventual connectivity of $U$ to its internal dynamics. If $c(U) < +\infty$, then $U$ has eventual connectivity two, and (by Lemma \ref{lem:ZP}) cannot be eventually isometric. It follows from Theorem \ref{thm:cutout} that $U$ is a trimodal domain with contracting and eventually isometric laminations $\mathcal{C}$ and $\mathcal{L}$ (respectively), and the uniqueness of the laminations together with Lemma \ref{lem:h} imply that $\mathcal{C}$ is made of level sets of $h$.

If $c(U) = +\infty$, we still have the contracting lamination given by Lemma \ref{lem:h}, and it remains to see what happens for pairs of points on different leaves. To this end, we invoke the Harnack metric $\chi_\Omega$ of a domain $\Omega$; it is defined as
\[ \chi_\Omega(z, w) := \sup \left\{\left|\log\frac{u(z)}{u(w)}\right| : \text{$u$ is positive and harmonic in $\Omega$}\right\}. \]
Its relevance to us lies on the fact that, if $\Omega$ is a bounded domain, it satisfies $\chi_\Omega(z, w) \leq d_\Omega(z, w)$ for all $z, w\in\Omega$ (see, for instance, \cite{BS06}). Now, since our positive harmonic function $h$ is $f$-invariant, we have for any pair $z, w\in U$ on different level sets of $h$ the following:
\[ \left|\log\frac{h\left(f^n(z); f^n(z_0), U_n\right)}{h\left(f^n(w); f^n(z_0), U_n\right)}\right| = \left|\log\frac{h(z; z_0, U)}{h(w; z_0, U)}\right| > 0. \]
Thus, for any $n\in\mathbb{N}$,
\[ \left|\log\frac{h(z; z_0, U)}{h(w; z_0, U)}\right| \leq \chi_{U_n}\left(f^n(z), f^n(w)\right) \leq d_{U_n}\left(f^n(z), f^n(w)\right); \]
this gives us a positive lower bound for $d_{U_n}\left(f^n(z), f^n(w)\right)$, which cannot be eventually constant by the Schwarz-Pick lemma. Indeed, Bergweiler, Rippon, and Stallard show in \cite[Theorem 1.7]{BRS13} that having infinite connectivity means that the $U_n$ meet infinitely many critical points of $f$, and so $f|_{U_n}$ cannot be a local hyperbolic isometry. This completes the proof.
\end{proof}

\section{A wandering domain with no eventual connectivity} \label{sec:ex}

In this section, we prove Theorem \ref{thm:ex}. The construction of $g$ is based on the ideas of Baker, Kotus and Lü \cite{BKL90}: we use approximation theory to construct a sequence of meromorphic functions $(g_n)_{n\geq 0}$ such that each $g_n$ is responsible for the $n$-th step in the orbit of $U$, and $g = \sum_{n\geq 0} g_n$ defines a meromorphic function satisfying our claims. Figure \ref{fig:model} illustrates what the construction is supposed to look like: in black, we have what we hope will be the wandering domains, and in grey we see the larger sets where the approximation takes place. The coloured arrows denote the action of our ``model map'', which $g$ is supposed to approximate and is made up of the following functions:
\begin{enumerate}
    \item $\tau$ is an exponential, mapping a vertical strip onto an annulus;
    \item $\psi$ is a re-scaled Joukowski mapping (see below), and takes a symmetric annulus around the unit circle onto an ellipse;
    \item $\phi$ is the function $z\mapsto (z - i)^{-1} + (z + i)^{-1}$. Its relevant properties are outlined below, but suffice to say it maps a subset of $\mathbb{D}$ conformally onto a vertical strip;
    \item $\sigma$ is an affine map.
\end{enumerate}
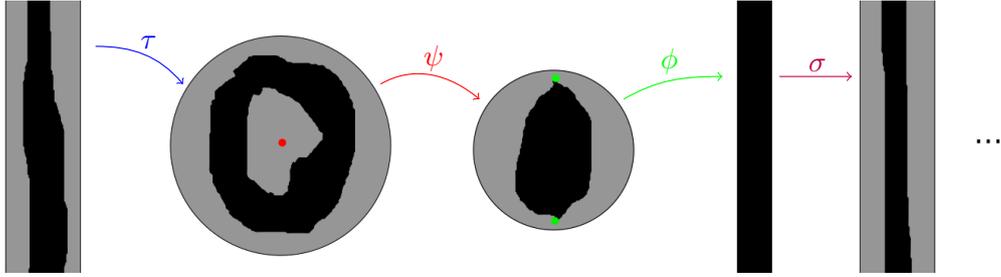
\begin{figure}[!h]
    \centering
    \begin{tikzpicture}
        \node[anchor=south west, inner sep=0] at (0,0)
        {\includegraphics[width=0.9\textwidth]{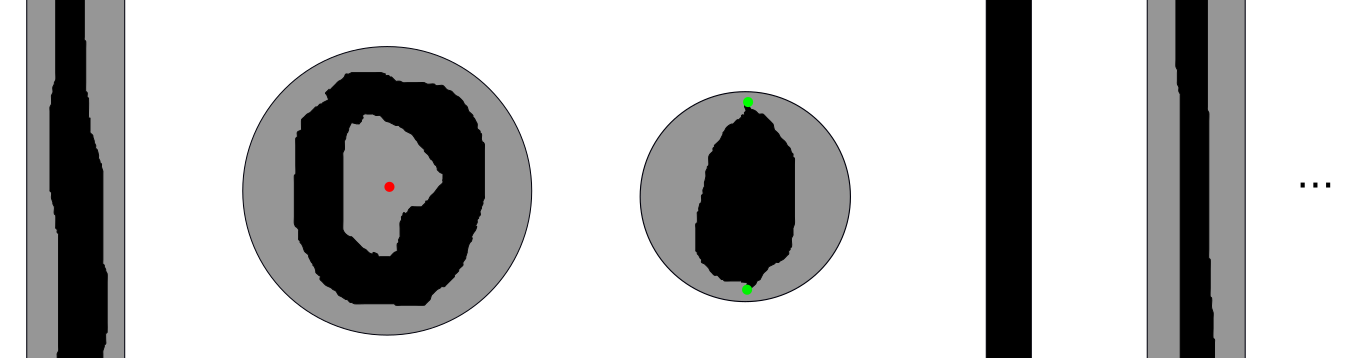}};

        \draw[->,blue] (1.45,3) to [out=0,in=130] (2.6,2.5);
        \node[blue] at (2.15,3.05) {$\tau$};

        \draw[->,red] (5.2,2.5) to [out=30,in=140] (6.5,2.3);
        \node[red] at (5.9,2.85) {$\psi$};

        \draw[->,green] (8.4,2.3) to [out=30,in=180] (9.7,2.6);
        \node[green] at (9,2.8) {$\phi$};

        \draw[->,purple] (10.45,2.6) to [out=0,in=180] (11.4,2.6);
        \node[purple] at (10.95,2.75) {$\sigma$};
    \end{tikzpicture}
    \caption{The first four steps of the model map, with approximating sets shown in grey and our intended wandering domains in black. The red and green dots denote the poles of our model map.}
    \label{fig:model}
\end{figure}

To this end, we must first show that the sequence $(g_n)_{n\geq 0}$ can be constructed to have the desired mapping properties, and then that the resulting Fatou components $U_n$ actually possess the desired topological characteristics. We start with some preliminary observations about the construction of the model map.

First, we choose a real constant $r\in (2, e)$, and a positive $\epsilon < 1/r$ (as we proceed, we will impose further restriction on $\epsilon$). Next, consider
\[ \phi(z) = \frac{1}{z - i} + \frac{1}{z + i} = \frac{2z}{z^2 + 1}; \]
this rational function of degree two has critical points $\pm 1$, which are also fixed points. More importantly for us, the vertical strip $\{z : |\Re z| < 1\}$ has a simply connected pre-image component (of degree one) under $\phi$ in $\mathbb{D}$, and this component touches the unit disc exactly at $\pm 1$ and $\pm i$. Since $r < e$, we know that $\phi^{-1}\left(\{z : |\Re z| < \log r - \epsilon\}\right)$ (of course, we take $\epsilon < \log r$) has a simply connected component $D\subset\mathbb{D}$ that touches $\partial D$ exactly at $\pm i$. Now, we want to choose a constant $\lambda\in\Cx$ so that the (re-scaled) Joukowski mapping
\[ \psi(z) = \lambda\left(z + \frac{1}{z}\right) \]
maps the annulus $\{z : 1/r < |z| < r\}$ into\footnote{In particular, since $D\subset\mathbb{D}$, $\lambda$ satisfies $|\lambda| < 1/2$.} $\{z\in D : \dist(z, \partial D) > \epsilon\}$, and then choose $R > r$ such that $\psi\left(\{z : 1/R < |z| < R\}\right)\supset \{z : 1 + \epsilon\}$, and some $\delta > \epsilon$; we will also require $R$ to satisfy other constraints to be specified ahead, but $\delta$ is relatively ``free''. For now, we make a few observations about the Joukowski mapping $z\mapsto z + 1/z$: it is a rational function of degree two, with critical points at $\pm 1$ and roots at $\pm i$; it maps the unit circle in two-to-one fashion onto the closed interval $[-2, 2]$, and is symmetric under $z\mapsto 1/z$. Consequently, $\psi$ maps annuli of the form $\{z : 1/C < |z| < C\}$, $C > 1$, properly and with degree two onto simply connected ellipses containing the origin.

With these preliminaries in place, we are ready to define our approximating sets. We will need sequences $l_n$ and $m_n$, $n\geq 0$, constructed as follows. Start with $l_0 = 1$, and then choose $m_0$ such that $m_0 - \log(2R') > l_0$, where the criteria for selecting $R' > R$ will be explained further ahead. The next value to be chosen is $l_1$, taken to satisfy $l_1 > m_0 + \log(2R')$. To choose $m_1$, we choose a different criterion, namely $m_1 - R > l_1$ and then for $l_2$ we want $l_2 > m_1 + R$. The next constant $m_2$ is chosen so that $m_2 - 1 - \delta > l_2$. For $l_3$, we pick a number so that $l_3 > m_2 + 1 + \delta$, and for $m_3$ we want $m_3 > l_3 + \log r + 1$. Finally, the last ``different'' step in this construction is $l_4$, which is taken so that $l_4 > m_3 + \log r + 1$. From now  on, the construction can be carried out recursively with the same basic structure as above (we start by treating $m_4$ as $m_0$), cycling between the rules with ``period'' four.

With these sequences ready, we define our sets as:
\[ F_n := \begin{cases}
            H_n\cup E_n = \{z : \Re z\leq l_n\}\cup \{z : |\Re z - m_n|\leq \log(2R')\}, & n\mod 4 = 0, \\
            H_n\cup E_n = \{z : \Re z\leq l_n\}\cup \{z : |z - m_n|\leq R\}, & n\mod 4 = 1, \\
            H_n\cup E_n = \{z : \Re z\leq l_n\}\cup \{z : |z - m_n|\leq 1 + \delta\}, & n\mod 4 = 2, \\
            H_n\cup E_n = \{z : \Re z\leq l_n\}\cup \{z : |\Re z - m_n|\leq \log r\}, & n\mod 4 = 3. \end{cases} \]

Let us make a few remarks about these sets: first, all of them are closed, and all of them contain a left half-plane and exactly one other component. If $n\mod 4 = 0$ or $n\mod 4 = 3$, these are vertical strips, albeit of different widths; if $n\mod 4 = 1$ or $n\mod 4 = 2$, these are discs of different radii. See Figure \ref{fig:sets} for the most ``relevant'' parts of these sets, and how they relate to each other.

\begin{figure}[!h]
    \centering
    \includegraphics[width=0.9\columnwidth]{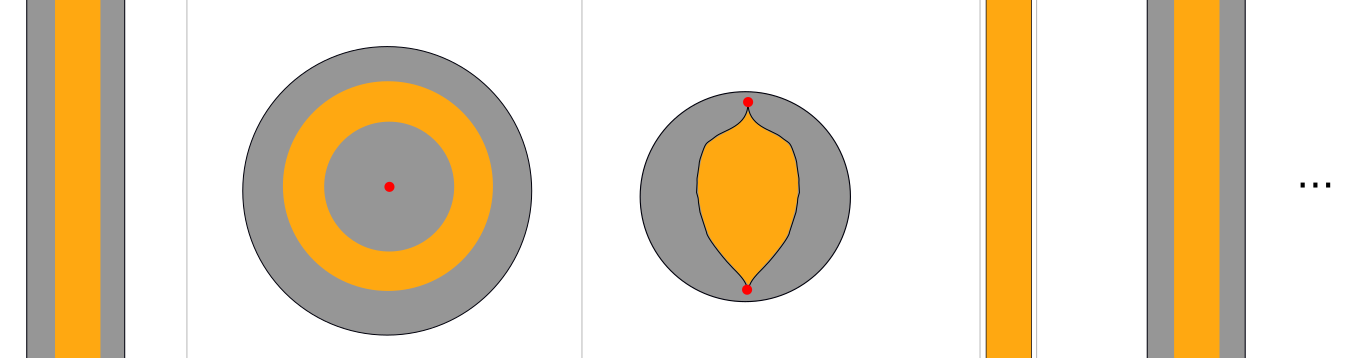}
    \caption{The ``main parts'' of the sets involved in this construction. Poles of the $g_n$, which are also the poles of $g$, are marked in red. In orange, we see some of the $G_n\subset E_n$, and, in light grey, the lines $L_n$.}
    \label{fig:sets}
\end{figure}

We will also need vertical lines $L_n := \{z : \Re z = x_n\}$. These are set up so that $m_n < x_n < l_{n+1}$ (except for $n\mod 4 = 2$, in which case we have $l_{n+1} < x_n < m_{n+1}$) and so that $L_n\cap F_m = \emptyset$ for any $n$ and $m$. We will reserve the right to fine-tune the position of some of these lines later.

We are ready to start our approximations. First, we must map $E_0$ into $E_1$ while simultaneously mapping $H_0$ into itself. To this end, we apply Arakelyan's theorem (see, for instance, \cite[~Section IV.C]{Gai87}) to the closed set $F_0\cup L_0$, obtaining an entire function $g_0$ satisfying
\begin{equation} \label{eq:aprox1}
  \begin{cases}
    |g_0(z)| < \epsilon_0^3, & z\in H_0, \\
    |g_0(z) - \exp\left((z - m_0)/2\right) - m_1| < \epsilon_0^3, & z\in E_0, \\
    |g_0(z)| < \epsilon_0^3, & z\in L_0 \end{cases}
\end{equation}
where $\epsilon_0$ is the first term in a sequence of positive numbers such that $\epsilon_n < \epsilon/10^n$.

Our next step is to approximate an appropriately translated version of $\psi$ in $E_1$. However, since $\psi$ has a pole, a simple application of Arakelyan's theorem will not suffice.

\begin{claim}
There exists a function $g_1\in\mathcal{M}(\Cx)$ with a single pole at $m_1$ such that
\begin{equation} \label{eq:aprox2}
 \begin{cases}
    |g_1(z)| < \epsilon_1^3, & z\in H_1, \\
    |g_0(z) + g_1(z) - \psi(z - m_1) - m_2| < \epsilon_1^3, & z\in E_1, \\
    |g_0(z) + g_1(z)| < \epsilon_1^3, & z\in L_1. \end{cases}
\end{equation}
\end{claim}
\begin{proof}
We apply Arakelyan's theorem to find an auxiliary entire function $h_1$ satisfying
\[ \begin{cases}
    |h_1(z) + \psi(z - m_1) + m_2| < \epsilon_1^3, & z\in H_1, \\
    |g_0(z) + h_1(z)| < \epsilon_1^3, & z\in E_1, \\
    |g_0(z) + h_1(z) + \psi(z - m_1) + m_2| < \epsilon_1^3, & z\in L_1, \end{cases} \]
and define $g_1$ as $g_1(z) = h_1(z) + \psi(z - m_1) + m_2$. It follows easily from the definition that $g_1$ has a single pole, which is at $m_1$, and satisfies the desired inequalities.
\end{proof}

By relying on a similar auxiliary function, we obtain the next step in our construction: a meromorphic function $g_2$ with poles at $m_2\pm i$ such that
\begin{equation} \label{eq:aprox3}
 \begin{cases}
    |g_2(z)| < \epsilon_2^3, & z\in H_2, \\
    |g_0(z) + g_1(z) + g_2(z) - \phi(z - m_2) - m_3| < \epsilon_2^3, & z\in E_2. \end{cases}
\end{equation}
Notice that we did not add the approximation on $L_2$ to the requirements; that will be up to our next function, which is the entire function $g_3$ satisfying
\begin{equation} \label{eq:aprox4}
 \begin{cases}
    |g_3(z)| < \epsilon_3^3, & z\in H_3, \\
    |g_0(z) + g_1(z) + g_2(z) + g_3(z) - (z - m_3)(\log r - \epsilon)/\log r - m_4| < \epsilon_3^3, & z\in E_3, \\
    |g_0(z) + g_1(z) + g_2(z) + g_3(z)| < \epsilon_3^3, & z\in L_2\cup L_3. \end{cases}
\end{equation}

Thus armed, we proceed inductively, constructing a sequence $(g_n)_{n\geq 0}$ of approximating meromorphic functions. Since the functions get progressively smaller on progressively larger left half-planes, the sum
\[ \sum_{n\geq 0} g_n \]
converges locally uniformly to a transcendental meromorphic function $g$ with infinitely many poles, which are exactly at $m_{4k+1}$ and $m_{4k+2}\pm i$ for $k\geq 0$. Furthermore, for any $z\in E_n$, $g$ never differs from the model map by more than
\[ \sum_{m\geq n} \epsilon_m^3\leq 1000\epsilon^3/999, \]
which -- if $\epsilon > 0$ is chosen sufficiently small -- is much smaller than $\epsilon$. If we define the sets
\[ G_n := \begin{cases}
            \{z : |\Re z - m_n| \leq \log r\}, & n\mod 4 = 0, \\
            \{z : 1/r < |z - m_n| < r\}, & n\mod 4 = 1, \\
            D + m_n, & n\mod 4 = 2, \\
            E_n, & n\mod 4 = 3, \end{cases} \]
it is clear from the definitions that $G_n\subset E_n$ for all $n\geq 0$. Hence, it follows by Equations (\ref{eq:aprox1}) to (\ref{eq:aprox4}) that $g(G_n)\subset G_{n+1}$, and so by Montel's theorem each $G_n$ belongs to a Fatou component $U_n$ such that $g(U_n)\subset U_{n+1}$. Furthermore, by the same token, $H_0$ belongs to an attracting Fatou component $V$, and every $L_n$ is contained in a pre-image of $V$. Therefore, the $U_n$ are distinct Fatou components (in particular, $U_0$ is a wandering domain with orbit $U_0, U_1, \ldots)$. It remains to show that each $U_n$ has the desired topological properties. To this end, we make two claims.

\begin{claim} \label{claim:bounded}
If $\epsilon > 0$ is chosen sufficiently small, then $\lambda$, $R$, $R'$, and $x_n$ can be chosen (``independently of $\epsilon$'') so that $U_n\subset E_n$ for $n\mod 4\in \{0,1,2\}$.
\end{claim}
\begin{proof}
This is where we will use the exact placement of the $L_n$ next to $E_{4k+3}$, $k\geq 0$, and where we will explain the further requirements on $R$ and $\epsilon$. An important point about this procedure is that, whenever a choice is made regarding $\lambda$, $R$, or $R'$, we can continue to shrink $\epsilon$ without affecting this choice. It is in this sense that we say that they are ``independent of $\epsilon$''.

First, we ask that $\log r + 2\epsilon < 1$. We now fix $\lambda\in\Cx$ as explained before: it is such that $\psi\left(\{z : 1/r < |z| < r\}\right)\subset \{z\in D : \dist(z, \partial D) > \epsilon\}$. Now, notice that (since $|\lambda| < 1/2)$ $\psi^{-1}\left(\{z : |z| < 1 + \epsilon\}\right)$ is a doubly connected domain surrounding the origin and bounded away from both zero and infinity. Therefore, there exists an annulus $A_R := \{1/R < |z| < R\}$ such that
\[ \text{$|\psi(z)| > 1 + \epsilon$ for every $z\in\partial A_R$,}  \]
and we can decrease $\epsilon$ without affecting the choice of $R$ as promised. We shrink $\epsilon$ so that $1/R > \epsilon$, and then pick $R'$ such that $1/R' + \epsilon < 1/R$ and $R' - \epsilon > R$.

After all that trouble, it follows that\footnote{This is essentially an application of Rouché's theorem; see \cite[Theorem 18]{Ahl79} and the following corollary.}, if $\gamma$ is a curve in an $\epsilon$-neighbourhood of $E_{4k+3}$, $k\geq 0$, we can ``pull it back'' by $g$ at least three times while remaining within the sets $E_{4k+i}$, $i = 0, 1, 2$ (see Figure \ref{fig:pullback}). Hence, if the values of $x_{4k+2}$ and $x_{4k+3}$ are such that $m_{4k+3} - x_{4k+2} < \log r + \epsilon$ and $x_{4k+3} - m_{4k+3} < \log r + \epsilon$, then $L_{4k+2}$ and $L_{4k+3}$ will be pulled back as shown in Figure \ref{fig:pullback}. Since every $L_n$ belongs to the orbit of the attracting domain $V$, all these pullbacks will be in the same attracting grand orbit, and thus cannot intersect any of the $U_n$. The conclusion follows.

\begin{figure}[!h]
    \centering
    \includegraphics[width=0.9\columnwidth]{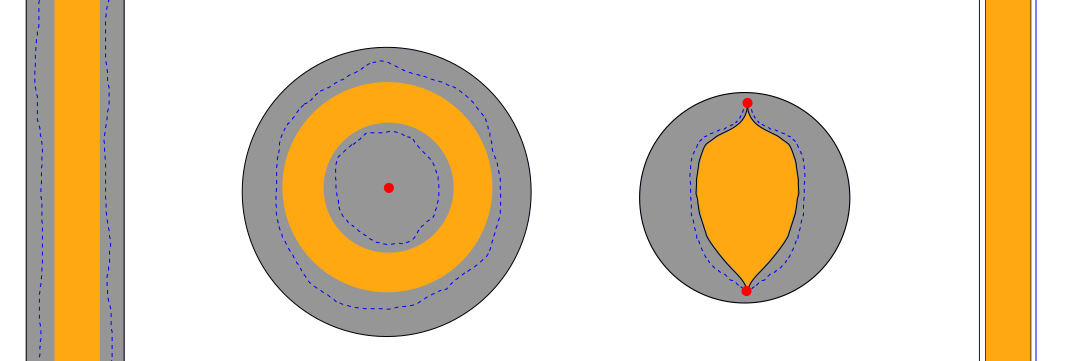}
    \caption{With (mostly) the same colour scheme as Figure \ref{fig:sets}, we have added pre-images of $L_2$ and $L_3$ in dashed blue ($L_2$ and $L_3$ are highlighted in solid blue). If the sets and constants are chosen as specified in Claim \ref{claim:bounded}, then these pre-images will persist even after the approximation.}
    \label{fig:pullback}
\end{figure}
\end{proof}

This suffices for showing the boundedness or unboundedness of the wandering domains $U_n$; next, we tackle the problem of connectivity.

\begin{claim} \label{claim:conn}
For any sufficiently small $\epsilon > 0$, the domains $U_{4k+2}$, $k\geq 0$, are simply connected.
\end{claim}
\begin{proof}
We will derive the argument for $U_2$, which is the case $k = 0$; any other $k$ follows an analogous argument. First, it is clear from the proof of Claim \ref{claim:bounded} that $g^{-1}(L_2\cup L_3)$ consists of two analytic arcs that connect the poles of $g$ at $m_2\pm i$, surrounding $G_2 = D + m_2$ and hence $U_2$. Therefore, no closed curve $\gamma\subset U_2$ can surround the poles of $g$; nevertheless, by \cite[Lemma 2]{RS08}, any such $\gamma$ that is not null-homotopic in $U_2$ must have an iterate $g^m(\gamma)$ that surrounds a pole of $g$. The only hope for $\gamma$, then, is to have an iterate in one of the Fatou components $U_{4k+1}$ for some $k\geq 1$ that surrounds the pole of $g$ at $m_{4k+1}$.

However, such an iterate would require $g^{m-1}(\gamma)$ to be a closed curve in $U_{4k}$, which must (since $g$ is holomorphic on $E_{4k}\supset U_{4k}$) surround a pre-image $z^*\in E_{4k}$ of $m_{4k+1}$. Still, by the triangle inequality (with $\tau(z) = \exp\left((z - m_{4k})/2\right) + m_{4k+1}$) and the error bound for $g$ on $E_{4k}$, we have
\[ \text{$|g(z) - m_{4k+1}| \geq |\tau(z) - m_{4k+1}| - |\tau(z) - g(z)| > 1/R' - \epsilon$ for every $z\in E_{4k}$.} \]
If we choose $\epsilon$ such that $\epsilon < (2R')^{-1}$, then the right-hand side of this inequality is greater than $(2R')^{-1}$, which means that no $z\in E_{4k}$ can reach $m_{4k+1}$ through $g$. We conclude that $U_2$ cannot contain any closed curve $\gamma$ that is not null-homotopic in $U_2$, and so $U_2$ is simply connected.
\end{proof}

A similar argument shows that the domains $U_{4k}$ and $U_{4k+3}$, $k\geq 0$, are also simply connected. Finally, the domains $U_{4k+1}$ are at least doubly connected, for each surrounds a pole of $g$ at $m_{4k+1}$; if the connectivity is greater than two, then (by Lemma \ref{lem:bolsch}) $U_{4k}$ is infinitely connected, which is a contradiction. This concludes the proof of Theorem \ref{thm:ex}.

\begin{remark}
It remains to be seen if the construction above can be simplified to yield a connectivity sequence of period less than four. However (and this was brought to my attention by Reem Yassawi), it \textit{can} be irregularly ``padded'' with strips and translations to yield a connectivity sequence that is not periodic at all.
\end{remark}

\section{Convergence to the boundary in multiply connected wandering domains} \label{sec:boundary}

In this section, we shall prove Theorem \ref{thm:classII}. Our starting point will be the same as in \cite{BEFRS19}: namely, the fact that for any hyperbolic region $D$ we have
\[ \rho_D(z)\to +\infty \Leftrightarrow \dist(z, \partial D) \to 0. \]

Since we will be working with $\widetilde{U}_n$, which is simply connected, we will also be able to use their elegant ``Harnack-type'' estimates \cite[Lemma 4.1]{BEFRS19}:

\begin{lemma} \label{lem:harnack}
Let $\Omega\subset \Cx$ be a simply connected hyperbolic domain. Then, for all $z$ and $w$ in $\Omega$,
\[ e^{-2d_\Omega(z, w)} \leq \frac{\rho_\Omega(z)}{\rho_\Omega(w)} \leq e^{2d_\Omega(z, w)}. \]
\end{lemma}
\begin{remark}
The proof of Lemma \ref{lem:harnack} applies the Koebe distortion theorem to the Riemann map from $\mathbb{D}$ to $\Omega$. As such, similar estimates \textit{can} be obtained for multiply connected domains by applying the distortion theorems in \cite{Zhe01b} to a universal covering map -- provided we assume that $\partial \Omega$ is uniformly perfect, and replace the universal factor of $2$ by a constant depending on $\Omega$. See also \cite{BP78} for other properties of the hyperbolic metric of simply connected domains that generalise to multiply connected ones under the assumption of uniform perfectness and domain-dependent constants.
\end{remark}

Now, notice that the first claim of Theorem \ref{thm:classII} (that all orbits go to the boundary or stay away from it together) follows from the second one (that, if an orbit approaches a certain sequence in $\partial U_n$, so does every other orbit). As such, we can prove Theorem \ref{thm:classII} by showing the following.

\begin{theorem} \label{thm:bndry}
Let $U$ be a wandering domain of a meromorphic function $f$. For $z\in U_n$, let
\[ \delta_n(z) := \mathrm{dist}\left(z, \partial \widetilde{U}_n\right). \]
Let $z_0\in U$ and, for $n\in\mathbb{N}$, let $w_n\in\partial \widetilde{U}_n$ be such that \[ \delta_n\left(f^n(z_0)\right) = |f^n(z_0) - w_n|. \]
If there exists a subsequence $(m_k)_{n\in\mathbb{N}}$ such that $\delta_{m_k}\left(f^{m_k}(z_0)\right)\to 0$, then $|f^{m_k}(z_1) - w_{m_k}| \to 0$ as $k\to +\infty$ for every other $z_1\in U$.
\end{theorem}
\begin{proof}
We will show that $|f^{m_k}(z_0) - f^{m_k}(z_1)|\to 0$, whence the theorem follows by the triangle inequality. To that end, let $\gamma_k\subset \widetilde{U}_{m_k}$ be a distance-minimising geodesic arc joining $f^{m_k}(z_0)$ to $f^{m_k}(z_1)$. Since it is distance-minimising, every $w\in \gamma_k$ satisfies
\[ d_{\widetilde{U}_{m_k}}\left((w, f^{m_k}(z_0)\right) \leq d_{\widetilde{U}_{m_k}}\left((f^{m_k}(z_1), f^{m_k}(z_0)\right) =: C_k, \]
and therefore
\[ C_k = \int_{\gamma_k} \rho_{\widetilde{U}_{m_k}}(z)\,|dz| \geq e^{-2C_k}\rho_{\widetilde{U}_{m_k}}\left(f^{m_k}(z_0)\right)\int_{\gamma_k}\,|dz| = e^{-2C_k}\rho_{\widetilde{U}_{m_k}}\left(f^{m_k}(z_0)\right)\ell_{\Cx}(\gamma_k) \]
by Lemma \ref{lem:harnack}, where $\ell_{\Cx}$ denotes the Euclidean length of a curve. This can be rearranged to yield
\[ |f^{m_k}(z_0) - f^{m_k}(z_1)| \leq \ell_{\Cx}(\gamma_k) \leq \frac{C_ke^{2C_k}}{\rho_{\widetilde{U}_{m_k}}\left(f^{m_k}(z_0)\right)}, \]
and we apply standard estimates on the hyperbolic density in simply connected domains \cite[Theorem 4.3]{CG93} to obtain
\[ |f^{m_k}(z_0) - f^{m_k}(z_1)| \leq 2C_ke^{2C_k}\delta_{m_k}\left(f^{m_k}(z_0)\right). \]
Now recall that, since $U_{m_k}\subset \widetilde{U}_{m_k}$, we have
\[ C_k = d_{\widetilde{U}_{m_k}}\left(f^{m_k}(z_0), f^{m_k}(z_1)\right) \leq d_{U_{m_k}}\left(f^{m_k}(z_0), f^{m_k}(z_1)\right) \leq d_U(z_0, z_1) \]
by the Schwarz-Pick lemma applied twice (first to the inclusion $\iota:U_{m_k}\to \widetilde{U}_{m_k}$, and then to $f^{m_k}:U\to U_{m_k}$). Thus,
\[ |f^{m_k}(z_0) - f^{m_k}(z_1)| \leq 2d_U(z_0, z_1)e^{2d_U(z_0, z_1)}\delta_{m_k}\left(f^{m_k}(z_0)\right), \]
and we see that the right-hand side goes to zero by hypothesis as $k\to+\infty$.
\end{proof}

It's a different question whether these different behaviours can \textit{occur} for multiply connected wandering domains. Here, as with the hyperbolic metric, we see remarkable rigidity.

\begin{cor}
Let $U$ be a Baker wandering domain of a transcendental meromorphic function $f$. Then,
\[ \liminf_{n\to+\infty} \dist\left(f^n(z), \partial\widetilde{U}_n\right) > 0 \]
for every $z\in U$.
\end{cor}
\begin{proof}
Assume that there exists some $z_0\in U$ for which this is not true (in fact, in order to keep notation light, we will assume that $\dist(f^n(z_0), \partial\widetilde{U}_n)\to 0$; it is easy to see how to adapt the argument for the case where this happens only on a subsequence). Then, by Theorem \ref{thm:classII}, the same is true for all $z\in U$; we take now a simple closed curve $\gamma\in U$ that is not null-homotopic in $U$. By the Schwarz-Pick lemma, we have $\ell_{U_n}(\gamma_n) \leq \ell_{U_n}(f^n\circ\gamma)\leq \ell_U(\gamma)$, where $\gamma_n$ stands for $f^n\circ\gamma$ traversed only once. Simultaneously, the estimates for hyperbolic density used in the proof of Theorem \ref{thm:bndry} yield
\[ \ell_{U_n}(\gamma_n) = \int_{\gamma_n} \rho_{U_n}(s)\,|ds| \geq \int_{\gamma_n} \frac{1}{2\dist(s, \partial\widetilde{U}_n)}\,|ds| \geq \frac{\ell_{\Cx}(\gamma_n)}{2\max\{\dist(s, \partial\widetilde{U}_n) : s\in\gamma_n\}}; \]
we see that the only way for the right-hand side to remain bounded\footnote{There may be some issues with the uniformness of the convergence of $\dist(f^n\circ\gamma, \partial\widetilde{U}_n)$ to $0$, but we can apply Egorov's theorem to find a positive measure subset of $\gamma$ (relative to arc-length measure) where convergence \textit{is} uniform.} is to have $\ell_{\Cx}(\gamma_n)\to 0$. This applies to every closed curve in $U$, and in particular to the curve surrounding the complementary component of $U$ containing the origin. This is a contradiction, for as $U$ is a Baker wandering domain we have $U_n\to \infty$ while surrounding the origin -- which clearly implies $\ell_{\Cx}(\gamma_n)\to +\infty$ at least for this particular curve.
\end{proof}


\begin{thebibliography}{26}
\providecommand{\natexlab}[1]{#1}
\providecommand{\url}[1]{\texttt{#1}}
\expandafter\ifx\csname urlstyle\endcsname\relax
  \providecommand{\doi}[1]{doi: #1}\else
  \providecommand{\doi}{doi: \begingroup \urlstyle{rm}\Url}\fi

\bibitem[1]{Ahl79}
L.~V. Ahlfors.
\newblock \emph{Complex Analysis}.
\newblock McGraw-Hill, 3rd edition, 1979.

\bibitem[2]{Bak76}
I.~N. Baker.
\newblock An entire function which has wandering domains.
\newblock \emph{J. Austral. Math. Soc. Ser. A}, 22:\penalty0 173--176, 1976.

\bibitem[3]{BS06}
H.~S. Bear and W.~Smith.
\newblock A tale of two conformally invariant metrics.
\newblock \emph{J. Math. Anal. Appl.}, 318:\penalty0 498--506, 2006.

\bibitem[4]{BM06}
A.~F. Beardon and D.~Minda.
\newblock The hyperbolic metric and geometric function theory.
\newblock In \emph{Proceedings of the International Workshop on Quasiconformal
  Mappings and their Applications}. New Delhi Alpha Science International,
  2006.

\bibitem[5]{BP78}
A.~F. Beardon and C.~Pommerenke.
\newblock The {Poincaré} metric of plane domains.
\newblock \emph{J. London Math. Soc.}, 18:\penalty0 475--483, 1978.

\bibitem[6]{BEFRS19}
A.~M. Benini, V.~Evdoridou, N.~Fagella, P.~J. Rippon, and G.~M. Stallard.
\newblock Classifying simply connected wandering domains, 2019.
\newblock Available at \url{https://arxiv.org/abs/1910.04802}. To appear in the
  \textit{Mathematische Annalen}.

\bibitem[7]{BRS08}
W.~Bergweiler, P.~J. Rippon, and G.~M. Stallard.
\newblock Dynamics of meromorphic functions with direct or logarithmic
  singularities.
\newblock \emph{Proc. London Math. Soc.}, 97:\penalty0 368--400, 2008.

\bibitem[8]{BRS13}
W.~Bergweiler, P.~J. Rippon, and G.~M. Stallard.
\newblock Multiply connected wandering domains of entire functions.
\newblock \emph{Proc. London Math. Soc.}, 107:\penalty0 1261--1301, 2013.

\bibitem[9]{Bol99}
A.~Bolsch.
\newblock Periodic {Fatou} components of meromorphic functions.
\newblock \emph{Bull. London Math. Soc.}, 31:\penalty0 543--555, 1999.

\bibitem[10]{Bus10}
P.~Buser.
\newblock \emph{Geometry and Spectra of Compact Riemann Surfaces}.
\newblock Birkhäuser, 2010.

\bibitem[11]{CG93}
L.~Carleson and T.~W. Gamelin.
\newblock \emph{Complex Dynamics}.
\newblock Springer, 1993.

\bibitem[12]{Com11}
M.~Commerford.
\newblock Short separating geodesics for multiply connected domains.
\newblock \emph{Cent. Eur. J. Math.}, 9:\penalty0 984--996, 2011.

\bibitem[13]{Fat20}
P.~Fatou.
\newblock Sur les équations fonctionnelles.
\newblock \emph{Bull. Soc. Math. France}, 48:\penalty0 208--314, 1920.

\bibitem[14]{Gai87}
D.~Gaier.
\newblock \emph{Lectures in Approximation Theory}.
\newblock Birkhäuser, 1987.

\bibitem[15]{Hay64}
W.~K. Hayman.
\newblock \emph{Meromorphic Functions}.
\newblock Oxford University Press, 1964.

\bibitem[16]{Hay76}
W.~K. Hayman and P.~B. Kennedy.
\newblock \emph{Subharmonic Functions}, volume~1.
\newblock Academic Press, 1976.

\bibitem[17]{BKL90}
J.~Kotus I.~N.~Baker and Y.~Lü.
\newblock Iterates of meromorphic functions {II}: examples of wandering
  domains.
\newblock \emph{J. London Math. Soc.}, 42:\penalty0 267--278, 1990.

\bibitem[18]{KH95}
A.~Katok and B.~Hasselblatt.
\newblock \emph{Introduction to the Modern Theory of Dynamical Systems}.
\newblock Cambridge University Press, 1995.

\bibitem[19]{KS08}
M.~Kisaka and M.~Shishikura.
\newblock On multiply connected wandering domains of entire functions.
\newblock In P.~J. Rippon and G.~M. Stallard, editors, \emph{Transcendental
  Dynamics and Complex Analysis}, pages 217--250. Cambridge University Press,
  2008.

\bibitem[20]{Lee13}
J.~M. Lee.
\newblock \emph{Introduction to Smooth Manifolds}.
\newblock Springer, 2nd edition, 2013.

\bibitem[21]{RS08}
P.~J. Rippon and G.~M. Stallard.
\newblock On multiply connected wandering domains of meromorphic functions.
\newblock \emph{J. London Math. Soc.}, 77:\penalty0 405--423, 2008.

\bibitem[22]{RS11}
P.~J. Rippon and G.~M. Stallard.
\newblock Slow escaping points of meromorphic functions.
\newblock \emph{Trans. Amer. Math. Soc.}, 363:\penalty0 4171--4201, 2011.

\bibitem[23]{Six13}
D.~J. Sixsmith.
\newblock On fundamental loops and the fast escaping set.
\newblock \emph{J. London Math. Soc.}, 88:\penalty0 716--736, 2013.

\bibitem[24]{Tsu75}
M.~Tsuji.
\newblock \emph{Potential Theory in Modern Function Theory}.
\newblock Chelsea Publishing Company, 2nd edition, 1975.

\bibitem[25]{Zhe01}
J.-H. Zheng.
\newblock \emph{Dynamics of Transcendental Meromorphic Functions}.
\newblock Tsinghua University Press, 2001{\natexlab{a}}.

\bibitem[26]{Zhe01b}
J.-H. Zheng.
\newblock Uniformly perfect sets and distortion of holomorphic functions.
\newblock \emph{Nagoya Math. J.}, 164:\penalty0 17--33, 2001{\natexlab{b}}.

\end{thebibliography}
\end{document}